\documentclass{amsart}
\usepackage{amsthm,amssymb,verbatim}
\usepackage{graphicx}
\usepackage{enumerate}

 \newcommand{\RR}{\mathbf{R}}  
 \newcommand{\ZZ}{\mathbf{Z}}  
 \newcommand{\BB}{\mathbf{B}}  
 \newcommand{\dist}{\operatorname{dist}}

 \newcommand{\eps}{\epsilon}

\swapnumbers  

    \newtheorem{theorem}    {Theorem}       [section]
    \newtheorem{lemma}      [theorem]       {Lemma}
    \newtheorem{corollary}  [theorem]     {Corollary}
    \newtheorem{proposition}       [theorem]       {Proposition}
    
    \newtheorem*{claim}{Claim}
    \theoremstyle{definition}
    \newtheorem{definition}  [theorem] {Definition}
    \theoremstyle{definition}
    \newtheorem{remark}   [theorem]       {Remark}

\begin{document}

\setlength{\baselineskip}{1.2\baselineskip}

\def\Ss{\mathcal S}
\def\sh{\sigma _{2h}}
\renewcommand{\thesubsection}{\thetheorem}
\newcommand{\Aup}{A^{\rm upper}}
\title{The geometry of genus-one helicoids}
\date{July 16, 2007.  Revised April 17, 2009}
\begin{abstract}
We prove: a properly embedded, genus-one minimal surface
 that is asymptotic to a helicoid and that contains two straight lines must intersect
 that helicoid precisely in those two lines. In particular, the two lines divide the
 surface into two connected components that lie on either side of the helicoid.  
We prove an analogous result for periodic helicoid-like surfaces.  We also give
 a simple condition guaranteeing that an immersed  minimal surface with finite genus
 and bounded curvature is asymptotic to a helicoid at infinity.
\end{abstract}
\keywords{complete embedded minimal surface, helicoid, variational methods}
\subjclass[2000]{Primary: 53A10; Secondary: 49Q05, 58E12}
\author{David Hoffman and Brian White}
\address{Department of Mathematics \\ Stanford University\\ Stanford,
CA 94305} \address{hoffman@math.stanford.edu,
white@math.stanford.edu}
\thanks{The research of the first author was supported by  the National Science Foundation,
Division of Mathematical Sciences under grant DMS-0139410. The
research of the second author was supported by the National
Science Foundation, Division of Mathematical Sciences
  under grant~DMS-0104049-NCX} 
  \maketitle
\section{Introduction and Statement of Results}

In this paper we consider properly immersed minimal surfaces
${\mathcal S}\subset\RR ^3$ that have one end asymptotic to the
helicoid and  genus equal to one. We will call such a surface a
nonperiodic genus-one helicoid. We are interested in embedded,
nonperiodic genus-one helicoids.  Without loss of generality, we
may assume that ${\mathcal S}$ is asymptotic to a vertical
helicoid whose axis is the $z$-axis, $Z$. If ${\mathcal S}$
contains $Z$ and one horizontal line we will refer to ${\mathcal
S}$ as a {\em symmetric}, nonperiodic genus-one helicoid. Schwarz
reflection about the lines on the surface provides the symmetries.
Without loss of generality we may assume that the horizontal line is the
$x$-axis, $X$, and that  $\Ss$ is asymptotic to the standard helicoid $H$, half of which is
parametrized by
\begin{equation}\label{helicoid}
(r,\theta)\rightarrow (r \cos \theta, r\sin \theta, \theta), \end{equation}$r\in[0,\infty)$, $\theta\in
\RR$. 
(The other half is obtained by Schwartz reflection about $Z$.)  
Note that $X\cup Z\subset H$.

 Hoffman, Weber and Wolf \cite{wwh1} proved the existence of a
symmetric, embedded,   nonperiodic  genus-one helicoid. In
\cite{hoffwhite1} we gave a variational construction for such 
surfaces. The examples we constructed in that paper have the
following property:
\begin{equation}\label{Ansatz}
\begin{aligned}
 &\text{$\Ss \cap H = X \cup Z$, and} \\
 &\text{$\Ss\setminus H$ consists of two congruent, simply connected components.}
\end{aligned}
\end{equation}
 In Theorem~\ref{nonperiodicdecomp} of
Section~\ref{halfHel}, we prove that {\em every embedded,
symmetric, nonperiodic genus-one helicoid satisfies
\eqref{Ansatz}}.

We also establish a parallel result for embedded {\em periodic}
genus-one helicoids, by which we mean properly embedded minimal
surfaces $S\subset \RR ^3$ that are invariant under a screw motion
\begin{equation}\label{screw}
 \sigma_{2h}(r \cos \theta,r\sin \theta,z)= (r\cos (\theta+2h), r\sin (\theta +2h), z+2h)
\end{equation}
 for some $h>0$, and for which $S/\sh$ has genus one
and is
asymptotic to $H/\sh$ at infinity.
Let $S^*=S\cap \{z| -h<z\leq h \}$, and
note that $S^*$ is a fundamental domain for $S$. If $S$ contains
 $Z$ and if $S^*$ contains {\em two} horizontal lines, then
we will refer to $S$ as a {\em symmetric}, periodic genus-one
helicoid. Without loss of generality, we may assume that $X\subset
S^{*}$, which implies that the other horizontal line in $S^{*}$ is $\sigma _h
(X)$. For $h>\pi/2$, such surfaces were proved to exist in \cite{hkw8,wwh1}, and by variational means
in  \cite{hoffwhite1}. Without loss of generality,
we may assume that $\{z=h\}\cap S =\sigma_h(X)$, which implies that 
$S^*$ is bounded by two lines. This follows, for example, from Lemma~1(vii)
of \cite{wwh1} together with an application of the maximum principle.
Define  $X^*=X\cup
\sigma _h (X)$.  The construction in \cite{hoffwhite1} produces
periodic surfaces satisfying an analog of \eqref{Ansatz} above:
\begin{equation}\label{Ansatz2}
\begin{aligned} 
   &\text{$S^*\cap H =X^*\cup Z^*$, and} \\
   &\text{$S^*\setminus  H$ consists of two
                      symmetric, simply connected components,}
\end{aligned}
\end{equation}
where $Z^*=\{(0,0,t)| -h<t\leq h\}$.
In this paper, we prove  that {\em every embedded, symmetric, periodic genus-one
helicoid satisfies \eqref{Ansatz2}.} This is
Theorem~\ref{periodicdecomp}.

We prove in Theorems~\ref{nonperiodicdecomp} and \ref{periodicdecomp} that embedded symmetric genus-one helicoids have simple intersections with all rotations of $H$:

{\em Let $\hat{H}$ be the result of rotating  $H$ about the
$z$-axis, $Z$, through an angle in $(0,\pi)$. If ${\mathcal S}$ is nonperiodic, then ${\mathcal S}\cap
\hat{H}$ consists of $Z$ together with a smooth embedded closed curve that intersects $Z$ twice,
once above and once below the $xy$-plane. If $S$ is periodic,
$S^*\cap \hat{H}$ consists of $Z^*$ together with a smooth embedded
closed curve that  intersects
$Z^*$ twice, once above and once below the $xy$-plane.}

Section~\ref{two}
concludes with a uniqueness result for 
half-helicoids, Theorem~\ref{halfhelicoiduniqueness1}:
 {\em
Suppose $M$ is a connected minimal surface that lies in the closure of a component
of $\RR ^3 \setminus H$, with $\partial M$ lying in the closure of a component, $\Sigma$, of $H\setminus Z$. If $M$ is bounded or if $M$ is asymptotic to  $\Sigma$, then $M\subset \Sigma$.}
The proof of this result uses the fact that  $\RR ^3 \setminus
Z$ is foliated by half-helicoids. Our approach is close to that
taken by Hardt and Rosenberg in \cite{HR1}.

As mentioned above, \cite{wwh1} and \cite{hoffwhite1} proved existence of $\sigma _{2h}$-invariant, symmetric genus-one helicoids for every $h>\pi/2$. 
In Theorem~\ref{nosmallh} of Section~\ref{nonexistence}, we prove
that the condition $h> \pi/2$ is necessary: for $h\le \pi/2$, there are no embedded, symmetric,
periodic genus-$g$ helicoids (with $g\ge 1$) invariant under  the screw
motion $\sh$.  To our knowledge, this  was first observed by Bill
Meeks  for $h<\pi/2$. Our result requires only the presence of two
horizontal lines in $S/\sh$ (no assumption that $S$ contains the
axis $Z$). This
proof uses Proposition~\ref{radialdecay}, which gives estimates of
the radial decay of the vertical distance between the end of a
symmetric, periodic genus-one helicoid and the end of a helicoid.
We use the same estimates to prove 
(Theorem~\ref{LastProposition}) that the $\sigma_{2h}$-invariant, helicoid-like surfaces
constructed in \cite{hoffwhite1} are  asymptotic to helicoids 
and thus are in fact periodic genus-one helicoids.

In Section~\ref{asymptoticbehavior}, we investigate the geometry
of properly immersed minimal surfaces with finite genus and one
end. With a few additional assumptions, we prove that such a
surface is asymptotic to a helicoid:

\vspace{0.1in}

{\em
Let ${\mathcal S}\subset \RR^3$ be a properly immersed minimal
surface with finite genus, one end and bounded Gauss curvature. 
Suppose that ${\mathcal S}$
contains $X\cup Z$, and  that one level set $\{x_3=c\}\cap {\mathcal  S}$  has precisely one divergent component and a finite number of singular points.  Then ${\mathcal S}$ is conformally a compact Riemann surface punctured in one point corresponding to the end, and that end is asymptotic to a helicoid.
}

\vspace{0.1in}

This is Theorem~\ref{onehelicoidalend}. This result gives another proof  that the genus-one surfaces constructed in
\cite{hoffwhite1} are asymptotic to the helicoid. The method of proof here is a slight generalization of the method used in that paper. (See Theorem 6.1 in \cite{hoffwhite1}.)

\section{Structural properties of symmetric genus-one helicoids}\label{two}

An embedded, nonperiodic genus-one helicoid ${\mathcal
S}\subset\RR ^3$ is a properly embedded minimal surface in $R^3$
that is asymptotic to the helicoid at infinity. Without
loss of generality, we will assume that ${\mathcal S}$ is
asymptotic to the helicoid $H$ defined in the first paragraph of
the Introduction. The surface $H$ is a right-handed helicoid that
contains $Z$ and $X$. We say that ${\mathcal S}$ is {\em
symmetric} if it contains $Z$ and $X$.
 Similarly, an embedded,  {\em periodic} genus-one
helicoid is a properly embedded minimal surface  $S\subset \RR ^3$
invariant under a screw motion \eqref{screw}, such that  $S/\sh$
 has genus one and two helicoidal ends.  
 We say that $S$ is {\em symmetric}
if $Z\subset S$, and the fundamental domain $S^*=S\cap\{-h<z\leq
h\}$ contains $X^*=X\cup\sigma_h(X)$.

In this section we will prove that embedded, symmetric genus-one
helicoids are cut by $H$ precisely along $X\cup Z$ into two
congruent simply connected domains.   We also prove the analogous result for  periodic genus-one helicoids.  The technique involves the
study of minimal surfaces with boundary lying in a half-helicoid.
\stepcounter{theorem}
\subsection{Removal of the axes results in two congruent, simply
connected domains}

We begin by showing that removal of $X\cup Z$ from  a properly
embedded, nonperiodic, symmetric genus-one helicoid produces two
congruent, simply connected domains. Similarly, removal of
$X^*\cup Z^*$ from a fundamental domain of a properly embedded,
{\em periodic}, symmetric genus-one helicoid also produces two
congruent simply connected domains.

\begin{lemma}\label{a1} Suppose ${\mathcal S}$ is a properly
embedded, nonperiodic symmetric genus-one helicoid. Then
${\mathcal S}\setminus Z$ and
 ${\mathcal S}\setminus X$ are annuli, and
${\mathcal S}\setminus (X\cup Z)$ is a pair of congruent, simply connected
domains.
\end{lemma}

\begin{lemma} \label{a2} Suppose $S$ is a properly embedded, periodic,
 symmetric genus-one
helicoid invariant under $\sh$.  Let $S^*=\{-h<z\leq h\}$ be a
fundamental domain of $S$, and $Z^*= Z \cap \{ -h < z \le h\}$. 
Then $S^*\setminus Z^*$
and $S^*\setminus X^*$ are annuli, and
 $S^*\setminus (Z^*\cup X^*)$ is a pair of
congruent, simply connected domains.
\end{lemma}

\begin{proof}[Proof of  Lemma~\ref{a1}]
The surface ${\mathcal S}$ is topologically a once-punctured torus.  (In fact by Theorem~\ref{onehelicoidalend},  ${\mathcal S}$
is conformally a  once-punctured torus, but we will not use that
here.)  Thus the one-point compactification $T={\mathcal S}\cup\{\infty\}$ is a torus
and $Z\cup\{\infty\}$ is a simple closed curve in $T$.  Removing a simple closed
curve from a torus either separates it into a disk and a
once-punctured torus, or else results in a single annulus.

The rotation $\rho _Z$ is an isometry of ${\mathcal S}$ that
leaves $Z$ invariant. Therefore, $Z$ cannot divide $T$ into a disk
and punctured torus, because these pieces would have to be
homeomorphic (by the involution $\rho _Z$). Thus the result is a
single annulus $A$.

For similar reasons, $X$,  viewed as a curve in $T$, is a simple
closed curve. The same argument shows that removal of $X$ from $T$
produces an annulus.

Note that $Z$ and $X$ cross at the origin and at the point at
infinity.  These two points are represented as points on the
boundary of the annulus $A=T\setminus Z$. The positive ray of $X$
is a simple curve in $A$ going from one boundary point to another.
If it went from one boundary component of $A$ to the same boundary
component, it would divide $A$ into two components, one a disk,
the other an annulus. But if there were two components, they would
be homeomorphic (by $\rho _X$). Thus, the positive ray of $X$ goes
from one boundary component of $A$ to the other, and removing it
results in a disk. Now removing the negative ray of $X$ divides
that disk into two disks.  The two disks are congruent since they are related by the isometry $\rho _X$
\end{proof}

\begin{proof}[Proof of Lemma~\ref{a2}]
Note that the one-point compactification of $S\cap \{ -h < z < h\}$ is a torus $T$.
The proof of Lemma~\ref{a2}
is exactly the same as the proof of Lemma~\ref{a1}, except that one replaces $\Ss$ and $Z$
by $S\cap \{ -h < z < h\}$ and $Z\cap \{-h < z < h\}$.
\end{proof}

 \stepcounter{theorem}\subsection{The decomposition
theorem for nonperiodic symmetric genus-one
helicoids}\label{halfHel}

 \vspace{0.2in}
 The helicoid $H$ divides $\RR ^3$ into
two simply connected regions.  Let $H^+$  be the region that
contains $Y^+$, the positive ray of the $y$-axis, and let  $H^-$
be the other region. The axis $Z$ is contained in the helicoid
$H$, and
 $H\setminus Z$ consists of two simply connected components, each of which
 is we will refer to as a {\it  half helicoid}.  We will denote by $\Sigma _0$ the
 half-helicoid that contains the $X^+$, the positive $x$-axis.

More generally, let us extend the definition of  a   half-helicoid
to include any surface obtained by rotating one of the components
of $H\setminus Z$ through some angle about $Z$.  Thus the
half-helicoids form a foliation of $\RR^3\setminus Z$. In particular,  rotating $\Sigma_0$
through angles in $(0,\pi)$ produces a foliation of $H^+$.

\begin{theorem}\label{nonperiodicdecomp}
Let $\mathcal{S}$ be an embedded,  nonperiodic, symmetric
genus-one helicoid. Then
\begin{enumerate}[\upshape 1.]
 \item $\mathcal{S}\setminus (X\cup Z)$ consists of two simply connected, congruent components
  $D$ and $D'$.
 \item\label{DisjointFromH}  $D$ and $D'$ lie in $\RR^3\setminus H$, one in $H^+$, the
 other in $H^-$.
 \item\label{IntersectionIsACurve} Let $\hat{H}$ be a helicoid obtained by rotating $H$ about $Z$
 through an angle in $(0,\pi)$.  Then $D\cap \hat{H}$ (resp. $D'\cap \hat{H}$) is a smooth embedded curve with one endpoint
 in $Z^+$ and the other endpoint in $Z^-$.
\end{enumerate}
\end{theorem}

Here $Z^+$ and $Z^-$ are the components of $Z\setminus\{O\}$ where $z>0$ and $z<0$,
respectively.

\begin{proof} Statement~1 is Lemma~\ref{a1}.
It remains to prove statements 2 and 3. We begin by observing that  by
assumption ${\mathcal S}$ contains  the axes $X$ and $Z$ and  
$\partial D=\partial D'=X\cup Z$. Let
$\rho _X$ and $\rho _Y$ denote rotations by $\pi$ about $X$ and
$Y$, respectively. These symmetries are orientation-reversing on
${\mathcal S}$, and it is easy to see that they interchange $D$
and $D'$. It follows that their composition, $\rho_Y= \rho
_X\circ\rho _Z$, rotation by $\pi$ about the axis $Y$, leaves $D$
and $D'$ invariant and preserves orientation on $\Ss$.

If $p\in \RR^3\setminus Z$, let $\Sigma(p)$ be the half-helicoid
that contains $p$. If $p\in Z\setminus \{0\}$, let $\Sigma(p)$ be
the half-helicoid with the property  that $\Sigma(p)$ and $D$ have
the same tangent half-plane at $p$.  Note that
\begin{equation}
 \text{\em If $p\in Z\setminus\{0\}$, then $D\cap \Sigma(p)$ contains a smooth
   curve, one of whose  endpoints is $p$.}
\end{equation}

Since $\overline{D}\setminus Z$ is simply connected, there is a
smooth function
\[
   \theta: \overline{D}\setminus Z \to \RR \
\]
such that
\begin{equation}\label{e:DefineTheta}
   q= (x,y,z) = (r\cos\theta, r\sin\theta, z)
\end{equation}
for $q\in \overline{D}\setminus Z$, where
$r=r(x,y,z)=\sqrt{x^2+y^2}$. We may normalize $\theta$ so that 
\begin{equation}\label{X+}
  \theta=0 \mbox{\,\, on\,\, }X^+,
\end{equation}
where $X^+$ is the positive $x$-axis.
It follows that for some integer $k$,
\begin{equation}\label{X-}
 \theta =(2k+1)\pi \mbox{\,\, on\,\, }X^-,
\end{equation} 
where $X^-=\rho_Y(X^+)$, the negative $x$-axis. 
Since $\overline{D}\setminus \{0\}$ is a smooth manifold with
boundary $X\cup Z\setminus \{0\}$, the function $\theta$ extends
smoothly to $Z\setminus\{0\}$.  Thus $\theta$ is a smooth function
on $\overline{D}\setminus \{0\}$.

Since $\theta=0$ on $X^+$ and $\theta=(2k+1)\pi$ on $X^-$,
extending $\theta$ to $0$ is somewhat problematic. We get around
that by using the geodesic completion $D^*$ of
$\overline{D}\setminus\{0\}$. Note that $D^*$ is
$\overline{D}\setminus\{0\}$ together with two points, $0^+$ and
$0^-$, which are the limits of $(x,0,0)$ as $x\downarrow 0$ and
$x\uparrow 0$, respectively. We let
\begin{equation}\label{thetalimits}
\theta(0^+)=0 \mbox{\,\,\,and\,\,\,} \theta(0^-)=(2k+1)\pi,
\end{equation}
which makes $\theta$ continuous on all of $D^*$.

We now use the $\rho _Y$ symmetry of $D$. Since 
$\rho _Y (x,y,z)
= (-x,y,-z)$, 
$$
\cos (\theta \circ\rho _Y)=-\cos \theta \,\,\,\mbox{and}\,\,\,
\sin (\theta \circ\rho _Y)=\sin \theta,
$$ from
which it follows that $\theta\circ\rho _Y$ and $-\theta$
differ by an odd multiple of $\pi$. In fact, since $\rho
_Y(X^+)=X^-$, it follows from \eqref{X+} and~\eqref{X-} that
\begin{equation}\label{theta1}
\theta \circ\rho _Y =-\theta +(2k+1)\pi.
\end{equation}

 Since $D$ is a disk and $\rho  _Y$ is an orientation-preserving isometric involution of
 D, $\rho _Y$ has a unique fixed point.
 That is,
 $D\cap Y$ is a single point.
  Without loss of generality, we can assume that this fixed point is in $H^{+}$.
 (Since $D'=\rho_Z(D)$, we can simply relabel the disks.)
 In particular, 
\begin{equation}\label{e:Y^-}
    D\cap Y^-=\emptyset,
\end{equation}
where $Y^-$ is the negative $y$-axis.
 \vspace{0.1in}
 Let
\begin{equation}\label{e:DefineF}
\begin{aligned}
  &F: D \to \RR \\
  &F: (x,y,z) \mapsto  \theta(x,y,z)-z.
\end{aligned}
\end{equation}
The function $F$ is connected to the geometry of our situation by
the following elementary observation:
\begin{equation} \label{whyF}
\text{ $F$ is constant on every half-helicoid.}
\end{equation}
Thus statements~\ref{DisjointFromH} and~\ref{IntersectionIsACurve} become
statements about the function $F$.  In particular, we will prove statement~\ref{DisjointFromH}
by proving that $F(D)=(0,\pi)$.

 Claim 1: $F$ has no local maxima or local minima on $D$.

\begin{proof}[Proof of Claim 1] Suppose $F$ has a local maximum or minimum at $p\in
D$. Let $\BB$ be a  ball centered at
 $p$,  small enough so that $\BB$ is disjoint from $Z$ and  that $D\cap \BB$ is connected.
 Then there is a unique continuous extension of $\theta$ to $D\cup\BB$ so that
    \eqref{e:DefineTheta} still holds.
Note that if we use this extended $\theta$ together with
\eqref{e:DefineF} to define $F$ on $\BB$,  then it follows from
\eqref{whyF} that
 \[
     \{ q\in \BB: F(q)=F(p)\}  = \Sigma(p)\cap \BB.
 \]
Thus $D\cap \BB$ lies in the closure of one the connected
components of $\BB\setminus \Sigma(p)$, and $D\cap \BB$ and
$\Sigma(p)\cap\BB$ are tangent at $p$.  By the maximum principle,
$D\cap \BB$ and $\Sigma(p)\cap \BB$ coincide.  By analyticity, all
of $D$ is contained in a helicoid, a contradiction.
\end{proof}

Claim 2: $F$ has no local maxima or local minima on $Z\setminus
\{0\}$.
\begin{proof}[Proof of Claim 2]
 Suppose $p\in Z\setminus\{0\}$.   Since $\Sigma(p)$ and
$\overline{D}\setminus\{0\}$ are minimal surfaces with boundary and since
they are tangent at $p$,  $D\cap \Sigma(p)$ contains a smooth
curve $C$ with $p$ as one of its endpoints.   Note that $F$ is
constant along $C$ by \eqref{whyF}.   By Claim~1, none of the points of $C$ is a
local maximum or local minimum of $F$.  Thus $p$ is neither a
local maximum nor a local minimum of $F$.
\end{proof}

Claim 3: Suppose $\alpha$ is not an integral multiple of $\pi$.
Then either $F^{-1}(\alpha)$ is empty, or it is a single smooth
curve with one endpoint on $Z^+$ and the other endpoint on $Z^-$.

\begin{proof}[Proof of Claim 3] Suppose $p\in C=F^{-1}(\alpha)$.
Then  $C\subset \Sigma(p)$.  Note that $D$ is asymptotic to $H$ at
infinity and $\Sigma(p)$ is not (because $\alpha$ is not an
integral multiple of $\pi$), so $C$ lies in a bounded region of
$\RR^3$.  Now $C$ cannot contain a closed curve, because then that
curve would bound a region in $D$, and $F$ would have an interior
maximum or minimum in that region, violating Claim~1. Thus each
connected component $T$ of $C$ has the structure of a tree whose
endpoints are on $X\cap Z$.  Since $F=0$ on $X^+$ and $F=(2k+1)\pi$
on $X^-$, the endpoints must be on $Z$.  Since $0$ and $(2k+1)\pi$
are the only subsequential limits of $F(p)$ as $p\to 0$, in fact
the endpoints of $T$ must be on $Z\setminus\{0\}$.

Now $T$ cannot have two endpoints on $Z^+$. For if it did, $T$
would contain a curve $\Gamma$ joining those endpoints, and that
curve together with the interval $I\subset Z$ joining the
endpoints would bound a region $U$ in $D$.   Since $F$ is constant
on $\Gamma$, $F$ would have a local maximum or a local minimum at
some point $p\in U\cup I$.  But that is impossible by Claims~1 and
2.

Thus $T$ has at most one endpoint on $Z^+$ and (by the same
reasoning) at most one endpoint on $Z^-$.  It follows that $T$ is
a smooth curve joining a point on $Z^+$ to a point on $Z^-$.

We have shown: $C=F^{-1}(\alpha)$ is a union of disjoint curves,
each of which joins a point in $Z^+$ to a point in $Z^-$.
Furthermore, there cannot be more than one such curve.  For if
there were two such curves $T_1$ and $T_2$, then those curves
together with a pair of intervals $I^+$ and $I^-$ (in $Z^+$ and
$Z^-$, respectively) would bound a region $U$.  Since $F$ is not
constant on $\bar{U}$, its maximum on $\bar{U}$ is greater than $\alpha$ or
its minimum on $\bar{U}$ is less than $\alpha$.  Thus $F$ has a local
maximum or a local minimum on $U\cup I^+\cup I^-$.  But that is
impossible by Claims~1 and 2.
\end{proof}

Claim 4: If $p\in Y^-$, then $D\cap \Sigma(p)=\emptyset$.

\begin{proof}[Proof of Claim 4] Note that $D\cap \Sigma(p)$ is the union of
$F^{-1}(\alpha)$ over all $\alpha$ that are congruent to $3\pi/2$
mod $2\pi$.  Thus if $D\cap \Sigma(p)$ were nonempty,  by Claim~3
it would contains a curve joining $Z^+$ to $Z^-$.  But any such
curve in $\Sigma(p)$ must cross $Y^-$, and $D$ does not contain
any points in $Y^-$.  Thus $D\cap \Sigma(p)=\emptyset$.
\end{proof}

By Claim 4, the set $F(D)$ does not contain any
  values equal to  $3\pi/2$ mod $2\pi$.   In
particular, it contains neither $-\pi /2$ nor $3\pi/2$. Since $D$
is connected, $F(D)$ is an interval. Since $F|D$ has no local
maxima or local minima, $F(D)$ is an open interval $(a,b)$.  From
\eqref{thetalimits},   $[a,b]$ contains $0$ and $(2k+1)\pi$. Thus
\begin{equation}\label{FD}
        \text{$F(D) = (a,b)$, where $-\pi/2\le  a  \le 0$ and $\pi \le b \le 3\pi/2$,}
\end{equation}
 and $k=0$.  In particular, from
\eqref{theta1} we have
\begin{equation}\label{theta2}
\theta \circ\rho _Y=-\theta +\pi.
\end{equation}

Claim 5: $a=0$ and $b=\pi$.

\begin{proof}[Proof of Claim 5] If $b\ne \pi$, then by \eqref{FD}, $\pi<b\le 3\pi/2$.    Let
$b_i\in (\pi/2,b)$ with $b_i\to b$.  Then
  $F^{-1}(b_i)$ is nonempty, and by Claim~4, it must contain
   a point $p_i$ in the $xy$-plane.  Note that
$p_i$ has the form
\[
   p_i = (r_i\cos b_i, r_i\sin b_i, 0)
\]
where $r_i>0$. By passing to a subsequence, we may assume that the
$r_i$ converge to a limit  $r\in [0,\infty]$.   Now $r$ cannot be
$0$ since $\theta(0^+)=0$ and $\theta(0^-)=\pi$.  Also, $r$ cannot
be a finite nonzero number since otherwise $F$ would attain its
maximum, contradicting Claim~1.

Finally, if $r_i\to\infty$, then $\dist(p_i, H) \to b-\pi>0$,
which is impossible since $D$ is asymptotic to $H$ at infinity.
The contradiction proves that $b=\pi$. 
The proof that $a=0$ is essentially the same. (It also follows from the $\rho _Y$ symmetry of $D$.)

\end{proof}
Since $F(D) = (0,\pi)$, it follows that $D$ intersects only those  
half-helicoids
produced by rotating $\Sigma_0$ through an angle in $(0,\pi)$. As  
observed in the paragraph just before the
statement of the theorem,
those half-helicoids foliate $H^+$, so $D$ lies in $H^+$.  It follows  
that
$D'=\rho_YD$ lies in $H^-=\rho_YH^+$.  This completes the proof of  
statement~\ref{DisjointFromH} of the theorem.

Statement~\ref{IntersectionIsACurve} of the theorem follows from Claim~3, together with
the fact that $F(D)=(0,\pi)$.

\end{proof}

\stepcounter{theorem}
\subsection{The decomposition theorem for periodic symmetric genus-one helicoids}
\label{periodic} There is an analogous result to
Theorem~\ref{nonperiodicdecomp} for periodic, symmetric genus-one
helicoids.
\begin{theorem}\label{periodicdecomp}
Let $S$ be a periodic, embedded, symmetric genus-one helicoid.
Then
\begin{enumerate}
 \item [1.]$S^*\setminus (X^*\cup Z^*)$ consists of two simply connected components $D$ and $D'$.
 \item [2.] $D$ and $D'$ lie in $\RR^3\setminus H$, one in $H^+$ the
 other in $H^-$.
 \item [3.] Let $\hat{H}$ be a helicoid obtained by rotating $H$ about $Z$
 through an angle in $(0,\pi)$.  Then $D\cap H'$ (resp.
 $D'\cap\hat{H}$)
   is a smooth embedded curve with one endpoint
 in $Z^+$ and the other endpoint in $Z^-$.
\end{enumerate}
\end{theorem}

The sets $S^*$ and $X^*$ are defined in the Introduction,
 just before \eqref{Ansatz2}.
The proof of Theorem 2.6 is a straightforward adaptation of the  
proof of Theorem~\ref{nonperiodicdecomp}.
(Statement 1 was already proved in Lemma 2.2.)

\stepcounter{theorem}
\subsection{A Half-helicoid uniqueness
theorem}\label{halfhelicoidsection}

We end this section with a uniqueness theorem for minimal
surfaces that lie in the closure of a component of $\RR ^3
\setminus H$, and that are either compact or asymptotic to $H$ at infinity.
\begin{definition}\label{asymptoticdef}
Let $S$ be an unbounded, oriented, embedded surface in $\RR^3$.
We say that another surface $M$ is {\it asymptotic to $S$ at infinity} provided
there is a domain $\Omega\subset S$ and a function 
$
    u: \Omega \to \RR
$
such that
\begin{equation}\label{condition}
    \lim_{|p|\to \infty} (|u(p)| + |Du(p)|) = 0
\end{equation}
and such that outside of a compact subset of $\RR^3$, the surface $M$ coincides with the graph
\[
    \{ p + u(p)\nu(p):  p\in \Omega \},
\]
where $\nu(p)$ is the unit normal to $S$ at $p$.
\end{definition}
\begin{remark}\label{asymptoticremark}
If $M$ and $S$ have compact boundaries and bounded principal  
curvatures, then the
$C^1$ condition \eqref{condition} follows (by an Arzela-Ascoli type  
argument) from the analogous
$C^0$ condition
$
     \lim_{|p|\to \infty} |u(p)| = 0.
$
\end{remark}

\begin{theorem}\label{halfhelicoiduniqueness1}
Let $\Sigma$ be one of the components of $H\setminus Z$.
Suppose
$M$ is a connected minimal surface in $\overline{H^+}\setminus Z$ such that
$\partial M\subset \overline{\Sigma}$ and such that $M$ is either bounded 
or asymptotic to $\Sigma$ at infinity.

Then $M$ is a subset of $\Sigma$.
\end{theorem}

\begin{proof} We may assume without loss of generality that $\Sigma
$ is the component of $H\setminus Z$
containing the positive $x$-axis. 
We may also assume that $M$ is bounded:
To see this, rotate $M$ about $Z$ through an angle $-\eps$ and intersect with $H^+$ to get a new minimal surface $M(\eps)$.  If $M$ were a counterexample to the theorem, then for sufficiently small 
 $\eps>0$, the surface $M(\eps)$ would also be a counterexample.  
 Furthermore, $M(\eps)$ is bounded since $M$ is asymptotic to $\Sigma$.
 
 Thus from now on we assume that $M$ is bounded.

Let $C$ be a closed, solid circular cylinder of finite height that contains $M$ and whose axis of symmetry is $Z$.
Let $\Gamma$ be the boundary of $C\cap\Sigma$.  Minimize area among disks in 
  $\overline{H^+}$ with boundary $\Gamma$ and with $M$ as an obstacle (i.e., among disks $\Delta$ such that $M$ is contained in the closed region bounded by $\Delta\cup (C\cap \Sigma)$.)
Call the resulting disk $D$.
  
Let $\Sigma'$ be the half-helicoid $H\setminus (Z\cup \Sigma)$.  Rotate $\Sigma'$ in $H^+$ until
it touches $D$ at an interior point or until it becomes tangent to $D$ at some point of the interval 
 $I :=\Gamma\cap Z$.  Call the resulting half-helicoid $\Sigma^*$.
 
Note that one of the following must occur:
\begin{enumerate}
 \item[1.] $\Sigma^*$ touches $D$ at an interior point of $D$.
 \item[2.] $\Sigma^*$ is tangent to $D$ at an interior point of $I$.
 \item[3.]$\Sigma^*$ is tangent to $D$ at an endpoint of $I$.
\end{enumerate}

In case 1, $D$ is contained in $\Sigma^*$ by the maximum principle.
In case 2, $D$ is contained in $\Sigma^*$ by the boundary maximum principle.
In case 3, $\Sigma^*=\Sigma$ (since $D$ and $\Sigma$ are tangent at the endpoints of $I$)
and thus $D$ is contained in $\Sigma^*$.
  
In all three cases, we have shown that $D$ is contained in $\Sigma^*$.  Since 
   $\partial D\subset \Sigma$, 
 this implies that $\Sigma^*=\Sigma$.  Since $M$ lies between $\Sigma$ and $\Sigma^*$, in fact
 $M$ is contained in $\Sigma$.
 \end{proof}

\section{Nonexistence of embedded, periodic, higher genus helicoids with small twist angles}
\label{NONEXISTENCE}

 In this section we study properly embedded,
periodic minimal surfaces invariant under a screw motion
$\sigma_{2h}$ and asymptotic to the helicoid. We will show that if
$h\leq\pi/2$ and the if the intersection of the surface  with some
 horizontal plane is  a  line,  then it
must be the helicoid.

\stepcounter{theorem}\subsection{The total curvature of
almost-helicoidal curves}

The curvature of  a space curve
 $\theta \mapsto c(\theta)$ is given by
     $\frac{|c\,'(\theta)\times c\,''(\theta) |}{|c\,'(\theta) |^3}$,
 and therefore the total
curvature from $\theta=0$ to $\theta=A$ is
\begin{equation}\label{tot_curv}
\int_c k\,ds
=
\int _0 ^A\frac{|c\,'(\theta)\times
c\,''(\theta) |}{|c\,'(\theta) |^2}\,d\theta.
\end{equation}
 Now suppose that
\begin{equation*}
c(\theta)=(r\cos \theta, r\sin \theta, \theta+f(r,\theta))
\end{equation*}
for some function $f(r,\theta)$. In the special case that
   $f$ is constant,
 the curve $c$  is a standard helix. More
generally, suppose that
  $\frac{\partial  f}{\partial\theta}$ and
   $\frac{\partial ^2 f}{\partial  \theta ^2}$
tend to zero uniformly as $r\rightarrow \infty.$
 Since
 $
  |c\,'(\theta)|^2 = r^{2} +  (1+\frac{\partial   f}{\partial\theta} )^{2}$
  \,\,{and}\,\,
$|c\,''(\theta)|^2  = r^2 +  ( \frac{\partial ^2 f}{\partial   \theta ^2} )^2$,
   it follows that
$$
     |c\,''(\theta)|<|c\,'(\theta)|
$$
for $r$
sufficiently large. Therefore
\begin{equation*}
\frac{|c\,'(\theta)\times c\,''(\theta)|}{|c\,'(\theta)|^2}
\leq
\frac{|c\,'(\theta)|    \,    |c\,''(\theta)|}{|c\,'(\theta)|^2}
             <
\frac{|c\,'(\theta)|^2 }{|c\,'(\theta)|^2}=1
 \end{equation*}
for $r$ sufficiently large. It follows  that the total curvature~\eqref{tot_curv} is
 strictly less than~$A$. We state this
observation as a lemma.

\begin{lemma}\label{lessthanA}
 Let $f(r,\theta)$ be a function defined for $r\ge R_0$ and  $\theta\in [0,A]$.  Suppose
 that
          $\frac{\partial f}{\partial \theta}$ and
$\frac{\partial ^2 f}{\partial \theta ^2}$  tend to zero uniformly as
$r\rightarrow \infty$.
 Then for every sufficiently large~$r$, the curve
 \[
     \theta \in [0,A] \mapsto ( r \cos \theta, r\sin\theta, \theta + f( r ,\theta))
 \]
has total curvature strictly less than $A$.
\end{lemma}

\stepcounter{theorem}\subsection{Nonexistence of examples with
$h\leq\pi/2$}\label{nonexistence}

We will use Lemma~\ref{lessthanA} to prove

\begin{theorem}\label{nosmallh} Let $S$ be a properly immersed minimal surface
 that lies in the slab $\{-h\leq z\leq h\}$ and that is bounded by the
 two lines $H\cap \{z=h\}$ and $H\cap \{z= -h\}$, where $H$ is the standard 
 helicoid~\eqref{helicoid}.
  Suppose
 that $S$ is asymptotic\footnote{See Definition~\ref{asymptoticdef}.}
   to $$H\cap \{-h\leq z\leq h\}.$$
    If $h\le \pi/2$, then 
 $
     S = H \cap \{ -h \le z \le h \}.
 $
\end{theorem}

As a corollary we have

\begin{theorem}\label{nosmallh1} Suppose that $S$ is a properly immersed minimal surface
invariant under a screw motion $\sigma _{2h}$ with $h\leq \pi /2$, and that $S$ is asymptotic
to $H$ as $r=\sqrt{x^2+y^2}$ tends to infinity.
 If  the intersection of $S$ with some horizontal plane is a line,
then $S=H$.
\end{theorem}
\begin{proof}[Proof of Theorem~\ref{nosmallh1}]
Without loss of generality we may assume that $S\cap \{z=-h\}$ is
a  line.  By the $\sigma_{2h}$ invariance, $S\cap \{z=h\}$ must also be a  line.
Since $S$ is asymptotic to $H$ away from $Z$, these lines must be the lines
$H\cap \{z=-h\}$ and $H\cap \{z=h\}$.  Thus $S\cap \{ -h \le z \le h\}$ satisfies the
hypotheses of Theorem~\ref{nosmallh}, from which we conclude that  $S=H$.
\end{proof}

 \begin{proof}[Proof of Theorem~\ref{nosmallh}] By assumption,
 outside of any sufficiently large cylinder
            $C_R=\{x^2+y^2\leq R^2,\,\, |z|\leq h\}$, 
 one end of $S$ is
 a graph of the form
 $$\{(r \cos \theta, r\sin \theta, \theta +f(r,\theta))\,
 : -h\leq \theta \leq h, \,\,\,\, r>R\},$$
 with $f(-h,r)=f(h,r)=0$. 
 (A similar discussion applies to the other end.  Indeed, after rotation by $\rho_Z$, the other
 end has the same form.)
 By Proposition~\ref{radialdecay}\eqref{beta_rate2}, 
     \begin{align*}
  \left|\frac{\partial f}{\partial \theta}\right|&=o(r^{-\beta}), \\
 \nonumber \left|\frac{\partial ^2 f}{\partial \theta ^2}\right|&=o(r^{-\beta})
\end{align*}
  for every $\beta<\pi/2h$.
 Since $h\leq \pi /2$, these estimates hold for every $\beta<1$.

Let $S_R = S \cap C_R$ be the portion of $S$ inside the cylinder $C_R$.
Note that $\partial S_R$ is an extremal curve: it lies on the boundary of the convex set
$C_R\cap \{|z|\le h\}$.

\vspace{.1in} {\bf Claim}. For $R$ sufficiently large, the total
curvature of  $\partial S_R$ is strictly less than $4\pi$.

 \begin{proof}[Proof of Claim.]
 The curve $\partial S_R$ consists of two line
 segments---one on the top and one on the bottom disk of  $\partial C_R$---and
 two  nearly helical curves on $\partial C_R$.
 There are four corners where the curves and the line segments meet
 orthogonally.
 
  By
 Lemma~\ref{lessthanA}, each of the two nearly helical arcs
 has total curvature strictly less than $2h$, provided $R$ is sufficiently large.   Thus 
 $\partial S_R$ has total curvature
 strictly less than
 \[
      2(2h) + 4(\pi/2) = 4h + 2\pi \le 4\pi
 \]
 since $h\le \pi/2$ and since each corner contributes $\pi/2$ to the total curvature.
 \end{proof}

We now  apply the following uniqueness result \nocite{EWW1}:

\begin{theorem} \label{unicity} A  smooth, extremal
 Jordan curve with total curvature at most
$4\pi$ bounds precisely one minimal surface and that minimal
surface is an embedded disk.
\end{theorem}
Meeks and Yau \cite{my2} prove that a smooth, extremal Jordan
curve either bounds two distinct embedded minimal disks or it
bounds a unique minimal disk and no other minimal surface of any
genus. Together with the result of Nitsche \cite{Nit1} that  a
smooth Jordan curve with total curvature not greater than $4\pi$
bounds a unique minimal disk, they arrive at
Theorem~\ref{unicity}.

The curve $\partial S_R$ we are dealing with has four corners that can be
smoothed in the surface with an arbitrarily small increase in total curvature, so
that the smoothed curve will also have total curvature strictly less than $4\pi$.
Hence by Theorem~\ref{unicity} (applied to the smoothed curve), $S_R$ is simply connected for all sufficiently large $R$, and therefore $S$ is
simply connected. 

Thus the surface obtained from $S$ by  repeated Schwarz reflection about
the boundary lines is a nonplanar, singly periodic, embedded and simply
connected minimal surface.  By a theorem  of Meeks and Rosenberg \cite{mr93},
the only such surface is the helicoid.
 \end{proof}

\begin{remark}  For $h<\pi/2$, it is also possible to prove
Theorem~\ref{nosmallh}, without recourse to Theorem~\ref{unicity},
as follows.  The decay estimates of
Proposition~\ref{radialdecay} can be used to show
that (for $R$ sufficiently large) the curve $\partial S_R$ projects
monotonically to the boundary of a  convex region $\Omega$ in the plane $\{x=0\}$.
(One shows that the curvature at each point of the projections of the perturbed
helical arcs is strictly positive.)
By a theorem of Rado (see for example Sections 398 and 400 in Nitsche \cite{ni2}),
$S_R$ must be a graph over $\Omega$.   In particular, $S_R$ must be simply
connected.

If $h=\pi/2$, the curve $\partial S_R$ still projects monotonically to the boundary
of a region $\Omega$ in the plane $\{x=0\}$.   However, convexity of $\Omega$ at the projections
of the corner points of $\partial S_R$ seems to be delicate.  In particular, the convexity does not
seem to follow from the decay estimates in Proposition~\ref{radialdecay}. 
\end{remark}

\section{Asymptotic behavior of symmetric, properly immersed  minimal surfaces with one end and finite topology.}
\label{asymptoticbehavior}

The goal of this section is to prove the following theorem:
\begin{theorem}\label{onehelicoidalend}
Let ${\mathcal S}\subset \RR^3$ be a properly immersed, nonplanar minimal
surface with finite genus, one end, and bounded curvature. Suppose
\begin{enumerate}
\item[1.]  ${\mathcal S}$
contains
 $X\cup Z$, and
 \item[2.]
for some value of $c$, $\{x_3=c\}\cap {\mathcal S}$ has precisely one divergent component.
\end{enumerate}
 Then ${\mathcal S}$ is conformally a once-punctured Riemann surface, and ${\mathcal S}$   is asymptotic to a helicoid. 
\end{theorem}

Note that any level set $M\cap\{z=c\}$ of a properly embedded minimal surface  $M\subset\RR^3$ can be decomposed uniquely as a union of connected  $C^1$, properly immersed curves, all intersections and self-intersections of which are transverse. The intersection points are precisely the points of tangency of $M$  and the plane  $\{z=c\}$.
Thus hypothesis 2 of Theorem~\ref{onehelicoidalend} is equivalent to: {\em there are only finitely many points of tangency of $\{z=c\}$ and $M$, and  $M\cap\{z=c\}$ can be written as the union of finitely many connected $C^1$ immersed curves, exactly one of which is not closed.}

\begin{remark}

Our proof of Theorem~\ref{onehelicoidalend} is similar to the proof in [HW, \S 6.1] that  
the surfaces constructed
in that paper are conformally punctured tori and are asymptotic to a  
helicoid at infinity.

\end{remark}

\begin{proof}[Proof of Theorem~\ref{onehelicoidalend}]
The surface $\Ss$ satisfies the hypotheses of the following theorem of 
Rodriguez and Rosenberg \cite{RodRos}:
 
{\em Suppose $\Ss \subset \RR^3$ is a properly immersed minimal  
surface with one end
and with bounded curvature.  Suppose also that at some level $x_3=c$,  
the intersection
$\{x_3=c\} \cap \Ss$ consists of finitely many curves with finitely many  
intersections.
Then $\Ss$ is of finite type, i.e., $\Ss$ is conformally a once-punctured  
Riemann surface, the puncture
corresponding to the end, and the one-forms $dg/g$ and $dh$ 
are meromorphic on the compact surface.}
 
 Here, $g$ is the stereographic projection of the  Gauss map from the north pole, and 
    $dh=dx_3+i\, dx^*_3$ is a holomorphic one form on $M$. (The function  $x_3^*$ is a harmonic conjugate of $x_3$; it is locally  
well-defined up to an additive
constant.)  Note that $dh$ is closed but is not, in general, exact.

 \vspace{.1in}
 
 {\bf Claim.} The one form $dh$ has a double pole at the puncture and no residue. The
 one form $dg/g$ also has a double pole at the puncture.
 \vspace{.1in}
 
 Assuming the claim, we  can complete the proof of the theorem by using the following result of Hoffman and McCuan \cite{HOMc03}:
 {\em
Let $E\subset \RR^3$ be a properly immersed, minimal annular end
that is conformally a punctured disk. Suppose that $dg/g$
and $dh $  both have double poles at the puncture and that $dh$
has no residue at the puncture. If $E$ contains a vertical ray and
a horizontal ray, then  $E$ is asymptotic
to a helicoid at infinity.} 

We apply this theorem to an end $E$ of ${\mathcal S}$ corresponding to a neighborhood of the puncture.
That $E$ contains the requisite rays follows from assumption 1 of the theorem.

 \begin{proof}[Proof of Claim]  By assumption 2, in a neighborhood of the puncture, the level curve
$ \{x_3=c\}\cap {\mathcal S}$ consists of two smooth curves emanating from  the puncture. Hence $dh$ has a pole of order two at that point. (That $dh$ must have a pole at the puncture follows from the maximum principle and the fact that ${\mathcal S}$ has one end.) Since $dh$ is holomorphic on $S$,  it follows from Stokes' Theorem that $dh$ has no residue at the puncture. In a (possibly smaller) neighborhood of the puncture, $dh$ can be assumed to have no zeros. Since $dh$ has a double pole at the puncture, the level curves
 $\{x_3=a\}\cap {\mathcal S}$, for any value of $a$ 
are embedded in this neighborhood. In particular, ${\mathcal S}$ has an embedded end.

 We claim that $ dg/g$ must have a pole at the puncture.  For suppose it  
does not.  Then $g$ has a
well-defined value at the puncture, and $g$ is meromorphic on the one-
point compactification of ${\mathcal S}$.
 Thus ${\mathcal S}$ has finite total curvature (\cite{oss}, \cite{os1} Chapter 9, \cite{HoffmanKarcherSurvey}, Section 2.3 ) and,  as observed above, an embedded end. Such ends are asymptotic to a plane or to an end of the catenoid. (\cite {sc1},  \cite{jm1}, \cite{HoffmanKarcherSurvey},  Section 2.3) On a catenoid $C$, $dh$ has a simple pole at an end (observe that the level curves   $\{x_3=a\}\cap C$ are circles that do not pass through the point corresponding to the end.) Therefore,  ${\mathcal S}$ is asymptotic to a plane, and since it has one end, the maximum principle implies that it is equal to a plane.  Since we are assuming that ${\mathcal S}$ is nonplanar,  the contradiction shows that $dg/g$ has a pole at the puncture.

We now determine the order of the pole of $dg/g$ at the puncture.  
First of all we will show that the order of the pole is even.
Note that on the compact Riemann surface $\mathcal{S}\cup\{\infty\}$,  
the number of zeros minus the number of poles (counting  
multiplicities) is even. (It is $2(1-m)$ where $m$ is the genus
of $\mathcal{S}$.)  Hence to show that the pole at the puncture has even  
order,
it suffices to show that
\begin{enumerate}[\upshape (i)]
\item\label{first} The number of poles of $dg/g$ on $\mathcal{S}$ is  
even.
\item\label{second} The number of zeros of $dg/g$ on $\mathcal{S}
\setminus \{0\}$
     is even.
\item\label{third} The origin (if it is a zero of $dg/g$) is a zero of even order.
\end{enumerate}
On ${\mathcal S}$, $dg/g$  has poles precisely at the zeros and poles  
of $g$. Along $Z$, $g$ is unitary so no poles of $dg/g$ occur there.  
By the $180^\circ$ rotation $\rho_Z$ about
$Z$, the zeros and poles of $g$ are paired, counting multiplicity.
Hence, $dg/g$ has an even number of poles on ${\mathcal S}$.
This establishes~\eqref{first}.
The zeros of $dg/g$ occur at branch points of $g$ (zeros of the Gauss  
curvature), and the multiplicity of the zero of $dg/g$ is equal to  
the branching order of $g$.
Except for the origin,
these zeros also occur in pairs (the point $p\in \mathcal{S}\setminus  
Z$ being paired with
$\rho_Z(p)$ and $q\in Z\setminus\{0\}$ being paired with $-q$.)
This proves~\eqref{second}.
The tangent plane at the origin is the plane $P$ given by $x_2=0$.
The $180^\circ$ rotations $\rho_X$ and $\rho_Z$ about $X$ and $Z$ are  
symmetries
of $\mathcal{S}$ and therefore also of $P\cap\mathcal{S}$.  Thus $P
\cap\mathcal{S}$
must have an even number, say $2k$, of curves passing through the  
origin.
The order of branching of $g$ at the origin is then $2(k-1)$.
   This establishes~\eqref{third}, completing the proof that $dg/g$  
has a pole at infinity of
   even order.

The principal curvature function of a minimal surface is given by the expression (\cite{HoffmanKarcherSurvey}, page 15)

\begin{equation} \nonumber
k=\sqrt{-K}=\frac{4|dg/g|/|dh|}{(|g| +1/|g|)^2}.
\end{equation}
Along $Z$, the tangent plane to ${\mathcal S}$ is vertical, so  $g$ is unitary and 

\begin{equation}\label{k}
           k = \frac{|dg/g|}{|dh|}.
\end{equation}

Since the curvature of  ${\mathcal S}$ is bounded by assumption and since (as we have already shown) $dh$ has a pole of order two at the end, we see from \eqref{k} that
$dg/g$ has a pole at infinity of order at most two. Since the order is even, it must be exactly two.

\end{proof}
\renewcommand{\qedsymbol}{}  
 \end{proof}

\section{Asymptotic behavior of symmetric, periodic, properly
embedded minimal surfaces with  finite topology and one end}

In  this section we give  estimates for the rate that a minimal
graph with certain helicoidal qualities actually converges to a
helicoid.  We will use the estimates  to prove that the periodic
examples constructed in \cite{hoffwhite1} are asymptotic to the
helicoid. The estimates of Proposition~\ref{radialdecay} were also
used in Section~\ref{NONEXISTENCE}, to prove that  periodic
examples with small twist angles do not exist.

\begin{proposition}\label{radialdecay}
Let $v(r,\theta)=\theta$, a function whose multigraph $S$ over $\RR
^2\setminus\{0\}$ is a half-helicoid of $H$.
Suppose   $S'$ is another minimal  multigraph of  a function $u$
over a
region of the form $$W_A=\{(r,\theta)| r\geq A,\,\,\,  |\theta |<h \}$$
with the property that  for $r\geq A$, $$u(r,\pm h) =v(r,\pm h)=\pm h.$$
Suppose further that  $S'$  has asymptotically vertical normals
as $R\rightarrow\infty$. If $w=u-v$ is bounded, then
\begin{equation}\label{beta_rate}
\begin{aligned}
      |w|&=o(r^{-\beta}),\\
      |Dw|&=o(r^{-(1+\beta)}), \\
      |D^2w|&=o(r^{-(2+\beta)})
\end{aligned}
\end{equation}
for any $\beta<\pi/2h$, where $D=(\frac{\partial}{\partial x_1},\frac
{\partial}{\partial x_2})$, and $x_1=r\cos \theta$, $x_2=r\sin \theta
$. In particular,
\begin{equation}\label{beta_rate2}
\begin{aligned}
    \left|\frac{\partial w}{\partial \theta}\right|&=o(r^{-\beta}), \\
    \left|\frac{\partial ^2 w}{\partial \theta ^2}\right|&=o(r^{-\beta})
\end{aligned}
\end{equation}
   for any $\beta<\pi/2h$.
\end{proposition}

The proof of Proposition~\ref{radialdecay} involves a
 Phragm\'{e}n-Lindel\"{o}f-type argument. To apply it, we first show that $w$ is the solution
 of
 a linear elliptic equation.

\begin{lemma}\label{Difference}
The vertical distance function between two minimal graphs
satisfies a linear elliptic equation
\end{lemma}
This is a special case of a well known result for quasilinear,
elliptic partial differential equations. (See Gilbarg-Trudinger,
\cite{GT}, Chapter 10.) For the reader's convenience we include a
proof of the lemma.
\begin{proof}[Proof of Lemma~\ref{Difference}]

A function $u$ whose graph is a minimal surface satisfies
$$Q(u)=(1+u_2^2)u_{11}+(1+u_1^2)u_{22}-(2u_1u_2)u_{12}=0.$$
 If $v$ is another function whose graph
is a minimal surface, then writing $w=u-v$,
\begin{align*}\label{QQ}
0&=Qu-Qv\\
&= (1+u_2^2)w_{11}+(1+u_1^2)w_{22}-(2u_1u_2)w_{12}  \\
      &\qquad +v_{11}(u_2^2-v_2^2)+v_{22}(u_1^2-v_1^2)-2v_{12}(u_1u_2-v_1v_2)   \\
&=
a_{ij}w_{ij}+[v_{11}(u_2+v_2)]w_2+[v_{22}(u_1+v_1)]w_1 \\
     &\qquad -v_{12}((u_2+v_2)w_1+(v_1+u_1))w_2
     \\
&= a_{ij}w_{i,j}
     +[v_{22}(u_1+v_1)-v_{12}(u_2+v_2)]w_1+[v_{11}(u_2+v_2)-v_{12}(u_1+v_1)]w_2  \\
&= a_{ij}w_{ij} +b_kw_k,
\end{align*}
 where
 \begin{equation}\label{coefficients}
 \begin{aligned}
 a_{11} &=1+u_2^2,\,\,\,a_{22}=1+u_1^2,\,\,\,a_{12}=a_{21}=u_1u_2\\
 b_1 &=v_{22}(u_1+v_1)-v_{12}(u_2+v_2)\\
 b_2 &=v_{11}(u_2+v_2)-v_{12}(u_1+v_1).
 \end{aligned}.
 \end{equation}
The operator $L$ defined by
\begin{equation}
Lw:= a_{ij}w_{ij} +b_k w_k=0. \label{Lwequals0}
 \end{equation}
is   elliptic and  linear (its coefficients do not depend
on $w$ or its derivatives), and $Lw=0$.
\end{proof}

Suppose $u_n$ and $v_n$ are sequences of solutions to the minimal surface equation on a domain $\Omega$.  Then $w_n =u_n-v_n$ satisfies $L_n w_n=0$, where  $L_n=L(u_n,v_n)$ is the
linear elliptic operator defined  in Lemma~\ref{Difference}. We will have need of the following result in the proof of Proposition~\ref{radialdecay}.
\begin{corollary}\label{limitL}
 If $u_n$ and $v_n$ converge uniformly to zero, then   $L_n$ converges
 smoothly to the Laplacian on compact subsets of  $\Omega$.
 \end{corollary}
\begin{proof} Let $f_i:\Omega \rightarrow {\bf R}$ be a sequence of solutions to the minimal surface equation that converge uniformly to zero, and let $K$ be a compact subset of $\Omega$. Fix a positive integer $k$ and  let 
$$\Lambda=\limsup _i(\max_{x\in K} |D^kf_i(x)|).$$
By passing to a subsequence if necessary we may assume that the $\limsup$ is a limit.
By \cite{GT}, Corollary 16.7, $||f_i||_{C^{k+1}}$ is uniformly bounded  on compact subsets of $\Omega$. Thus, by passing to a further subsequence if necessary, we may assume that the $f_i$ converge, on compact subsets in the $C^k$ norm, to a limit function $f$. But since the $f_i$ converge uniformly to zero, $f$ is the zero function. Hence $\Lambda =0$.
 
 We now apply the conclusion of the previous paragraph to the sequences $u_n$ and $v_n$. 
From \eqref{coefficients}, it follows that 
$L_n$ converges smoothly on compact subsets to the Laplacian. 
\end{proof}

\begin{proof}[Proof of Proposition~\ref{radialdecay}]
Let $w=u-v$ be the difference between the two functions that
define the minimal multigraphs on $W_A$.  The function $w$ is zero on the
rays in $\partial W_A$ and
  bounded   above  on  the circular arc of radius $A$ in $\partial
 W_A$. By Lemma~\ref{Difference}, $w$ satisfies $Lw=0$ for the linear elliptic operator
 $L$ defined in (\ref{Lwequals0}).

For any  $0<\beta<\alpha<\frac{\pi}{2h}$,
 the functions
 \begin{align*}
           f(r,\theta)      &=        r^{-\beta} \cos ( \alpha  \theta )   \\
           {g}(r,\theta)  &=        r^{\beta} \cos( \alpha \theta )
  \end{align*}
are both positive on $\overline{W_{A}}$, and they  satisfy
\begin{equation}\label{Deltafg}
 \begin{aligned}
          \Delta f    &=   (\beta ^2-\alpha ^2)r^{-2}f  <   0,
                     \\
           \Delta g  &=   (\beta ^2-\alpha ^2)r^{-2}g  <  0
\end{aligned}
\end{equation}
on $W_A$.

\vspace{.2in}

\begin{claim}
There exists an $A'\geq A$, such that
$$Lf<0\,\, \mbox{and}\,\, L{g}<0 \,\,\,\mbox{on}
\,\,\,W_{A'} =\{(r,\theta)| r\geq A', \,\, |\theta|<h\}.$$
\end{claim}

\begin{proof}[Proof of Claim:]
 On $S$, the half helicoid  that is the
multigraph of $v(r,\theta) =\theta$, the normal is asymptotically vertical
as $r\rightarrow \infty$, and by assumption the same is true for
$S'$, the graph of $u(r,\theta)$. Therefore,  $|D u|\rightarrow
0$ and $|D v|\rightarrow 0$ as $r\rightarrow\infty$. Moreover
it is elementary to calculate that $|Dv|=O(r^{-1})$ and $|D^2
v|=O(r^{-2})$.

From \eqref{coefficients}, \eqref{Deltafg}, and the preceding
paragraph, we may compute
\begin{eqnarray} \label{Lf1} Lf &=&\Delta f +u_2^2
 f_{1,1} +u_1^2 f_{2,2}-2u_1 u_2 f_{1,2}+b_1f_1+b_2 f_2 \\&=&
Cr^{-2}f+o(1)(f_{1,1} + f_{2,2} -2f_{1,2}) +O(r^{-2})(f_1
+f_2),\nonumber
\end{eqnarray}
where $C=(\beta ^2 -\alpha ^2)<0$, and the $b_k$ are defined in
\eqref{coefficients}.  It is straightforward to compute that
$|Df|=O(r^{-1})f$ and $|D^{2}f|=O(r^{-2}) f$,
 from which it follows that
\begin{eqnarray}\label{Lf3}
Lf &=& Cr^{-2}f+o(1)O(r^{-2})f +O(r^{-3})f,\nonumber\\
&=& f[Cr^{-2} +o(1)O(r^{-2})+O(r^{-3})].
\end{eqnarray}
Since $f>0$ on $W_A$ and $C<0$, it follows from \eqref{Lf3} that
$Lf<0$ for $r$ sufficiently large.

An almost identical proof establishes that $Lg<0$ for $r$
sufficiently large.

\end{proof}

\vspace{.2in}
 Let $A'$ be the constant whose existence is
established by  Claim~1.  Since $w$ is bounded on
$\overline{W_{A'}}$ and $f$ is strictly positive on
$\overline{W_{A'}}$, there exists a $\lambda>0$ large enough so
that $\lambda f > w$ on the circular arc in $\partial W_{A'}$
of radius $A'$. On the rays in $\partial W_{A'}$, $w$ is
identically zero and both $f$ and $g$ are strictly positive. Therefore, for any
$\epsilon>0$, 
$$\tilde{w}_\epsilon :=w-\lambda f-\epsilon g< 0\,\,\,\,\mbox{on}\,\,\,\partial W_{A'}.$$

Also, for any $\epsilon >0$, $\tilde{w}_\epsilon (r,\theta)$ is negative for $r$ sufficiently large. (This is because $w$ is bounded,  $f\rightarrow 0$ and $g\rightarrow \infty$ as $r\rightarrow \infty$.)
 From  Claim~1 we have
\begin{equation} L\tilde{w}_\epsilon>0 \,\,\,\mbox{on}\,\,\,W_{A'}. \nonumber
\end{equation} 
By the maximum principle, we may conclude that $\tilde{w}_\epsilon < 0 $ on 
$W_{A'}
$. Since this is valid for any positive $\epsilon$, it follows that $w\leq\lambda f$ on 
$W_{A'}$, 
which implies that for any $\beta <\frac{\pi}{2h}$, $$w\leq \lambda r^{-\beta}.$$
 We can repeat the same argument
for $-w$ and conclude that  $-w\leq \lambda r^{-\beta}$.  We
have proved  the first equation of (\ref{beta_rate})
 for any $\beta <\frac{\pi}{2h}$. In particular, $S$ is
asymptotic to $S'$.

It remains to establish the asymptotic decay rates, stated in
\eqref{beta_rate}, of $Dw$ and $D^2w$. If a surface is minimal,
then its image under rescaling of $\RR ^3$ is also minimal.
Therefore \begin{equation} u_R(p):= \frac{u(Rp)}{R}
 \,\,\,\mbox{and} \,\,\, v_R(p):= \frac{v(Rp)}{R} \label{uRandvR}
 \end{equation} 
are solutions of 
 minimal surface equation defined for $r>A/R$ and $|\theta|\le h$.
 By Schwartz reflection in the boundary rays, we may extend $u_R$ and $v_R$
 to the region defined by $r>A/R$ and $|\theta|\le 2h$.
 Moreover, $u_R\rightarrow 0$ and $v_R\rightarrow 0$ as $R\rightarrow \infty$. By  Lemma~\ref{Difference} and Corollary~\ref{limitL}, $w_R=u_R-v_R$ satisfies a linear elliptic equation, 
 $L_Rw_R=0$, and as $R\rightarrow \infty$, $L_R$ converges smoothly to the Laplacian on compact subsets of the region where $r>0$ and $|\theta|<2h$.

Let $\Gamma$ be the arc given by $r=1$ and $|\theta|\le h$.  
Let $\Omega$ be the open set defined by $1/2 < r<2$ and $|\theta| \le \frac32 h$.
  We
will apply the Schauder interior estimates (\cite{GT}, Theorem
6.2) for the operator $L_R$ on $\Omega$. Since $L_R$ is converging
smoothly on $\Omega$, there exist  positive
constants $C$ and $R^{*}$ such that for $R>R^{*}$:
\begin{equation}\label{SchauderEst}
\begin{aligned}
 \sup_{p\in\Gamma}|Dw_R(p)|   &\le C \sup_{p\in\Omega}|w_R(p)|  \\
\sup_{p\in\Gamma}|D^2w_R(p)|    &\le   C \sup_{p\in\Omega}|w_R(p)|.
\end{aligned}
\end{equation}

Since $w_R(p)=\frac{w(Rp)}{R}$,
the first and second derivatives of $w_R$ satisfy
 \begin{equation}\label{scaling}
 \begin{aligned}
 Dw_R (p)  &=      Dw (Rp)   \\
 D^2w_R(p)   &= RD^2w(Rp).
\end{aligned}
\end{equation}

 Let 
 $R\Gamma := \{ Rp: p\in \Gamma\}$ and
 $R\Omega:=\{Rp\,\,|\,\, p\in\Omega\}$. It now follows from
\eqref{SchauderEst} and  \eqref{scaling} that for $R>R^{*}$:
\begin{equation}\label{SchauderEst2}
\begin{aligned}
     \sup_{q\in R\Omega}|Dw(q)|  &\le  R^{-1}C  \sup_{q\in R\Omega}|w(q)|\\
     R\sup_{q\in R\Omega}|D^2w(q)|&\le R^{-1}C  \sup_{q\in R\Omega}|w(q)|.
\end{aligned}
\end{equation}
 Since we have already established that $w =o(r^{-\beta})$, the last two
estimates of \eqref{beta_rate} follow immediately from
\eqref{SchauderEst2}.  The bounds~ \eqref{beta_rate2} follow directly from \eqref{beta_rate}. 
\end{proof}

We will use  Proposition~\ref{radialdecay} to prove this section's main result, 
which concerns  the asymptotic behavior of ends of  
embedded screw-motion-invariant minimal surfaces with some of the properties of the surfaces constructed in \cite{hoffwhite1}.

\begin{proposition}\label{LastProposition}
Let $S\subset \RR^3$ be a properly embedded minimal surface that is invariant
under a screw motion
\begin{equation}\label{screwyou}
   \sigma: (r\cos\theta, r\sin\theta, \theta) \mapsto (r\cos(\theta+\beta), r\sin(\theta+\beta), \theta+t)
\end{equation}
with $t\ne 0$.
Suppose that 
  the slope of the tangent plane tends to $0$ as $r=\sqrt{x^2+y^2}$ tends to infinity,
  that the intersection of $S$ with some horizontal plane coincides with a line $L$ outside
  of a compact set, and that
  the two ends of the line $L$ correspond to different ends of $S/\sigma$.
  
Then $S$ is asymptotic as $r$ tends to infinity to a helicoid $H'$ with axis $Z$.
Indeed, the vertical distance between $S$ and $H'$ decays faster than
$r^{-\beta}$ for every 
\[
    \beta < \left| \frac{p}{t}\right| \, \pi,
\]
where $p$ is the pitch of the helicoid $H'$.
\end{proposition}

The pitch of $H'$ may be defined as the (constant) value that 
  ${\partial z}/{\partial \theta}$ 
takes on $H'\setminus Z$.
Note that if $\beta=t=2h$, then $\sigma$ is the screw motion
$\sigma_{2h}$ used elsewhere in this paper~\eqref{screw}.

\begin{remark}
The theorem is also true for surfaces $S$ with boundary, provided the boundary lies
within a bounded distance of $Z$.
\end{remark}

\begin{proof}
By translation and rotation, we may assume that the line $L$ is $X$.
We will allow $S$ to have boundary as indicated in the remark.
Let $C$ be the interior of a closed solid cylinder about $Z$ that contains
$\partial S$.
Choose the radius $R$ of the cylinder large enough that
 the tangent planes to $S\setminus C$
are all nearly horizontal, and 
 so that $S\cap \{z=0\}$ coincides with $X$ outside of $C$.
 We may assume that $S\subset \RR^3\setminus C$ and that
$\partial S\subset  \partial C$; otherwise replace $S$ by $S\setminus C$.

Let $X^+$ and $X^-$ denote the positive and negative portions of $X\setminus C$:
\begin{align*}
  X^+ &= \{ (x,0,0): x \ge R\}  \\
  X^- &= \{ (x,0,0): x \le -R\}.
\end{align*}
Thus we have
\begin{equation}\label{ZeroSlice}
      S\cap \{z=0\}  =  X^+ \cup X^-.
\end{equation}

Note that each component $V$ of $S$ is a covering space of
$\RR^2\setminus \BB(0,R)$.   
Since $V$ is embedded, it must be either a single-sheeted covering or an infinite
covering.  Thus $V$ can be parametrized as
\begin{equation}\label{parametrization}
   (r,\theta) \mapsto ( r\cos\theta, r\sin\theta, f(r,\theta)) \qquad (r\ge R,\, \theta\in\RR)
\end{equation}
where this map is either periodic (with period $2\pi$) or else one-to-one, according to whether
$V$ is a single-sheeted or not.
Note that if $f$ is not periodic, then by properness $f$ is not bounded above
or below.

Now let $V$ be the component of $S$ containing $X^+$.  
By the hypothesis about the ends of $L$, $V$ cannot be the component of $S$ that contains
$X^-$, nor can it be the component that contains
any of the rays identified with $X^-$ in $S/\sigma$.  Thus
\begin{equation}\label{WrongRays}
    \text{$V\cap  \sigma^n X^- = \emptyset$ for $n\in \ZZ$.}
\end{equation}
Note that we can parametrize $V$
as in~\eqref{parametrization} with a function $f$ satisfying
\[
    f(r, 0) \equiv 0.
\]
%
%
%
%
%
%
%
%
%
%
By Schwartz reflection, $f(r,-\theta)\equiv - f(r,\theta)$.  In particular,
$f(r,-\pi) = - f(r,\pi)$.  Now $f(r,\pi)\ne 0$ since $V$ is disjoint from $X^-$.
Thus $f(r,\pi)\ne f(r,-\pi)$, so $f(r,\theta)$ does not have period $2\pi$
and therefore is  not periodic.

It follows (from~\eqref{ZeroSlice}) that $f(r,\theta)=0$
if and only if $\theta=0$.

%
%
%
%
%
%
%
%
%
%
%
   
 By reflecting in the plane $\{z=0\}$, if necessary, we may assume that
$f(r,\theta)>0$ for $\theta>0$.  We may also assume that  $t>0$.  (Otherwise
replace $\sigma$ by $\sigma^{-1}$.)
Since $f$ is not bounded above, it must intersect the plane $\{z=t\}$.
 By~\eqref{ZeroSlice} and by the $\sigma$-invariance,
 \[
     S\cap \{z=t\} = \sigma X^+ \cup \sigma X^-.
 \]
 Thus $V\cap \{z=t\}$ must be one or both of the rays $\sigma X^{\pm}$.  
 By~\eqref{WrongRays}, it must be the ray $\sigma X^+$.
 Let $\Theta$ be the value of the parameter $\theta$ corresponding to this ray.
 Then we have:
\begin{align*}
    &0< f(r,\theta) < t  \qquad \text{for $0<\theta<\Theta$},  \\
    &\text{$f(r,0) \equiv 0$, and $f(r, \Theta) \equiv t$.}
\end{align*}
If we dilate the surface $S$ by $\lambda>0$, then $R$, $t$, and $f(r,\theta)$
get replaced by $\lambda R$, $\lambda t$, and $f(r/\lambda,\theta)$, but $\Theta$ does not change.

Thus by scaling by $\lambda = \Theta/t$,  we can assume that $t=\Theta$.
Hence $\sigma=\sigma_{2h}$ (see~\eqref{screw}), where $h=t/2$.
Now $V$ (or more precisely $\sigma_{-h}V$) satisfies the
  hypotheses of Proposition~\ref{radialdecay}.
 Hence $V$ is asymptotic to $H$ with the asserted decay rate.
 
Let $W$ be the component of $S$ containing $X^-$.  Exactly the same argument shows that $W$
is also asymptotic (with the asserted decay rate) to 
some helicoid $H'$ with axis $Z$ and containing
 $X$.
Note that $H'=H$ since otherwise $V$ and $W$ would intersect.

Finally, $S$ can have no component
other than $V$ and $W$, because any such component, being
trapped between $V$ and $W$, would have to intersect the plane $\{z=0\}$,
and by~\eqref{ZeroSlice} the only possible intersections are $X^+$ and $X^-$.
\end{proof}

\nocite{MR01}\nocite{ColdingMInicozzi3} \nocite{ColdingMinicozziPNAS}
\bibliography{genusone}

\bibliographystyle{alpha}
\end{document}